\documentclass[12pt,oneside,reqno]{amsart}
\usepackage{amsmath}
\usepackage{graphicx}
\usepackage{mathrsfs}
\usepackage{stmaryrd}
\usepackage{amsfonts}
\usepackage{color}
\usepackage{enumerate,amsmath,amssymb,amsthm}
\newcommand{\be}{\begin{eqnarray}}
\newcommand{\ee}{\end{eqnarray}}
\newcommand{\ce}{\begin{eqnarray*}}
\newcommand{\de}{\end{eqnarray*}}
\newtheorem{theorem}{Theorem}[section]
\newtheorem{lemma}[theorem]{Lemma}
\newtheorem{remark}[theorem]{Remark}
\newtheorem{definition}[theorem]{Definition}
\newtheorem{proposition}[theorem]{Proposition}

\newtheorem{Example}[theorem]{Example}
\newtheorem{corollary}[theorem]{Corollary}

\def\[{{\Big[}}
\def\]{{\Big]}}
\def\<{{\langle}}
\def\>{{\rangle}}
\def\({{\Big(}}
\def\){{\Big)}}

\def\bx{{\mathbf{x}}}

\def\bt{\begin{theorem}}
\def\et{\end{theorem}}
\def\bl{\begin{lemma}}
\def\el{\end{lemma}}
\def\br{\begin{remark}}
\def\er{\end{remark}}
\def\bx{\begin{Example}}
\def\ex{\end{Example}}
\def\bd{\begin{definition}}
\def\ed{\end{definition}}
\def\bp{\begin{proposition}}
\def\ep{\end{proposition}}
\def\bc{\begin{corollary}}
\def\ec{\end{corollary}}

\def\geq{\geqslant}
\def\leq{\leqslant}

\allowdisplaybreaks

\begin{document}
\title{Stochastic Volterra integral equations with jumps and non-Lipschitz coefficients}
\author{Anas Dheyab Khalaf$^{1,2}$ and Xiangjun Wang$^{1}$}


\date{}
\dedicatory{$^{1}$ School of Mathematics and Statistics,
Huazhong University of Science and Technology,\\
Wuhan, Hubei 430074, P.R.China\\
$^{2}$ Department of Mathematics, college of computer science and mathematics,\\
Tikrit University, Tikrit, Iraq}


\thanks{CONTACT Anas Dheyab Khalaf. Email: anasdheyab@hust.edu.cn}
\begin{abstract}
Stochastic Volterra integral equations with jumps (SVIEs) have become very common and widely used in numerous branches of science, due to their connections with mathematical finance, biology, engineering and so on. In this paper, we apply the successive approximation method to investigate the existence and uniqueness of  solutions to the SVIEs driven by Brownian motion and compensated Poisson random measure under non-Lipschitz condition.
\end{abstract}

\maketitle

\section{introduction}
Stochastic differential equations (SDEs) with jumps offer the most flexible, numerically accessible mathematical framework to help model the evolution of financial and other random quantities through time.
In particular, feedback effects can be easily modeled and jumps enable us to frame events.
It is important to be able to incorporate event driven uncertainty into a model, and this can be expressed by jumps.
This arises, for instance, when one works on  credit risk, insurance risk, or operational risk.
SDEs enable us  to model the  feedback effects in the presence of jumps and independence on the level of
the state variable itself \cite{Bjork:1997, Geman:2006, Glasserman:1997, Schonbucher:2003}.

The Poisson process is usually used When one needs to handle the mathematical description of jumps, due to as its  advantage of counting events and generating an increasing sequence of jump times related to each event that it counts, and consequently, gives the number of jumps that have occurred up to any point in time. However, it is more appropriate to use Poisson measures driven-SDEs when events that have randomly distributed jumps are modeled, instead of Poisson processes, for complete exposition on this topic, we refer to the monographs  \cite{Elliott:1982,Rudiger:2004,Mandrekar:2015,Ren:2011,Albeverio:2010,Platen:2010}  for discussion of recent developments.
\\ In real life, many phenomena can be mathematically formulated by SVIEs. Therefore,  in past decades, SVIEs have attracted the attention of many scholars, who have investigated  the analytical and numerical solutions of such kind of equations, the main references are  \cite{Atkinson:1997,Brunner:1982, Brunner:2004, Tsokos:1974, Liang Hui:2017, Szynal:1988, Anas}.

Recently, \cite{Nacira:2018} studied the controlled stochastic Volterra integral equations with jumps \cite{Nacira:2018}, under a Lipschitz condition, they proved the existence and uniqueness of solutions to SVIEs with jumps under Lipschitz condition. Moreover, in \cite{Anas2} the authors of this paper joined Abouagwa M. and Almushaira M. and established the existence and uniqueness of solutions for the SVIEs under Lipschitz condition, and they provided numerical solutions for such equations by applying the Euler--Maruyama scheme.

Motivated by the aforementioned works, In the present paper, we apply the method of successive approximation, Bihari's inequality and Doob's martingale inequality to investigate  the existence and uniqueness of the solutions to the SVIEs driven by Brownian motion and pure jump L\`{e}vy motion with non-Lipschitz coefficients.

To finish this introduction, we note the general structure of this paper. In Section 2, we introduce some preliminaries and give a brief insight about the SVIEs, moreover we include some requisite lemmas. Section 3 is devoted to prove the existence and uniqueness theorem, lastly we give an example of the SVIE which has a unique solution.

\section{Some preparations }
Throughout this paper, unless otherwise specified, we let $(\Omega,\mathcal{F},\{\mathcal{F}_{t}\}_{t\geq{0}}  , P)$ to be a filtered probability space satisfying the usual conditions (i.e., it is increasing and right continuous while $\mathcal{F}_{0}$ contains all $P$-null sets).
Let $|\ .\ |$ be the Euclidean norm on $\mathbb{R}^{n}$.
\\Consider the stochastic Volterra integral equation with jumps of the form
\begin{align}\label{1}
\nonumber
x(t)=\varphi(t)+&\int_0^t{f(t,s,x(s))}ds+\int_0^t{g(t,s,x(s))}dW_{s}\\
+&\int_0^t\int_{R_{0}}{h(t,s,x(s),\xi)}\tilde{N}(ds,d\xi),
\end{align}
where the initial state $\varphi(t)$ is a given $\mathcal{F}_t$-adapted C\`{a}dl\`{a}g process, $E|\varphi(t)|^2<\infty$ for $t\in[0,T]$, $0<T<\infty$.
$ W(t)$ is Brownian motion, $\tilde{N}(dt,d\xi):=N(dt, d\xi)-\upsilon(d\xi)dt$ is a compensated Poisson random measure, $\upsilon$ represents the L\'{e}vy measure of the jump counting measure $N$, and $R_{0}=\mathbb{R}^n-\{0\}$. We impose  that $P(\int_0^t\int_{R_{0}}{|c(t,s,\xi)X(s)|^2}\upsilon(d\xi)ds<\infty)=1$.

Next, we give some requisite lemmas. the following lemma is taken from  \cite{AA}.
\begin{lemma}[Bihari's inequality]
Let $T>0$, $y_0\geq 0$, and $y(t), z(t)$ be continuous functions on $[0,T]$.
Assume that $\alpha:\mathbb{R}_+ \rightarrow \mathbb{R}_+$ ($\mathbb{R}_+$ is the set of of all nonnegative real
numbers) is a concave continuous non--decreasing function such that $\alpha(v)>0$ for $v>0$.
If
\begin{align}
\nonumber
y(t)\leq y_0+\int_0^t{z(s)\alpha(y(s))}ds \;\;\;\;\; \forall t\in[0,T],
\end{align}
then
\begin{align}
\nonumber
y(t)\leq G^{-1}(G(y_0)+\int_0^t{z(s)}ds) \;\;\;\;\;\;\; \forall t\in[0,T],
\end{align}
so that $(G(y_0)+\int_0^t{z(s)}ds)\in Dom(G^{-1})$, where $G(v)=\int_0^v{\frac{ds}{\alpha(s)}}$, $v>0$.
Moreover, if $y_0=0$ and $ \int_{0^+}{\frac{ds}{\alpha(s)}} = \infty$, then $y(t) = 0$ for all $t\in[0,T]$.
\end{lemma}
\begin{lemma}[Doob's martingale inequality]
If $ \{X(t)\}_{t\geq{0}}$ is a positive submartingale, then for any $p>1$ and for all $t>0$,
\begin{align}
\nonumber
E\left[\sup_{0\leq s \leq t}|X(s)|^p\right]\leq \left(\frac{p}{p-1}\right)^p E[|X(t)|^p].
\end{align}
\end{lemma}
We refer to  \cite{Applebaum:2009}, theorem 2.1.5, to obtain the proof.
 \\

In the rest of this work, the following assumptions are imposed on the coefficients of Eq.~(\ref{1}). 

 \begin{enumerate}
\item[\textbf{$H_1$}\;-]  There exists a positive constant $C$ such that for all $0\leq s \leq t \leq T$, $ x\in \mathbb{R}^n $
\begin{equation}
\nonumber
|a(t,s)x|^2\vee|b(t,s)x|^2\vee\int_{R_{0}}|c(t,s,\xi)x|^2\upsilon(d\xi)\leq C(1+|x|^2)
\end{equation}
\item[\textbf{$H_2$}\;-] Let $a,\; b$ and $c$ be measurable real--valued functions, for each fixed $t \in [0,T]$ and for all $x,y \in \mathbb{R}^n$, there exist two  functions  $\lambda$ and $\kappa$, such that
\begin{align}
\nonumber
|f(t,s,x)-f(t,s,y)|^2&\vee|g(t,s,x)-g(t,s,y)|^2\\
\nonumber
&\vee\int_{R_{0}}|h(t,s,x,\xi)-c(t,s,y,\xi)|^2\upsilon(d\xi)\\
\nonumber
&\leq \lambda(t)[\kappa(|x-y|^2)].
\end{align}
where $\lambda\in L^2([0,T];\mathbb{R})$, and $\kappa(.)$ is  monotone non-decreasing, continuous and is a concave function, satisfying  $\kappa(0)=0$ and $\kappa(v)>0$, such that
\begin{align}
\nonumber
\int_{0^+}{\frac{dv}{\kappa(v)}}=\infty.
\end{align}

\end{enumerate}

\section{The main result}
In this section, we apply some technical tools from Stochastic integration with respect to compensated Poisson random measures and L\`{e}vy noise (see \cite{Elliott:1982,Rudiger:2004, Anas, Oksendal:2007}, for complete picture on this point) together with Lemma \textcolor[rgb]{1,0,0}{1} and Lemma \textcolor[rgb]{1,0,0}{2}, to provide  a systematic proof for the existence and uniqueness solution to Eq.~(\ref{1}) under non--Lipschitz condition. This is the massage of the next theorem.
\begin{theorem}Suppose that Assumptions \textcolor[rgb]{1,0,0}{$H_1$-$H_2$} hold. If  $\varphi$ is a monotonically increasing, then Eq.~(\ref{1}) has a unique solution.
\end{theorem}
In order to prove Theorem \textcolor[rgb]{1,0,0}{3.1}, we define a sequence of successive approximations $\{x^{k}, k= 1,2,...\}$ with $x^0(t)=\varphi(t)$, as follows:
 \begin{align}\label{2}
\nonumber
x^{k}(t)=\varphi(t)+&\int_0^t{f(t,s,x^{k-1}(s))}ds+\int_0^t{g(t,s,x^{k-1}(s))}dW_{s}\\
+&\int_0^t\int_{R_{0}}{h(t,s,x^{k-1}(s),\xi)}\tilde{N}(ds,d\xi), \;\;t\in[0,T], \;\;k= 1,2,...
\end{align}
\begin{lemma} With assumptions as in  Theorem \textcolor[rgb]{1,0,0}{3.1}, for all $t\in [0,T]$, there exists a positive constant $C_1$, such that
\begin{align*}
E|x^k(t)|^2\leq C_1, \;\;k=1,2,...
\end{align*}
\end{lemma}

\begin{proof}Let $\tilde{T}:= max\{T,1\}$. Here, for $k=1,2,...$, we shall show that
\begin{align} \label{3}
E|x^k(t)|^2\leq 4E|\varphi(T)|^2\sum_{\ell=0}^{k}{\frac{(4C\tilde{T})^{\ell}}{\ell!}t^{\ell}}+\sum_{\ell=1}^{k}{\frac{(4C\tilde{T})^{\ell}}{\ell!}t^{\ell}}.
\end{align}

First, for $k=1$, from Eq.~(\ref{2}) and using the following simple inequality
 \begin{align}\label{4}
|x_{1}+x_{2}+...+x_{m}|^2\leq m(|x_{1}|^2+|x_{2}|^2+...+|x_{m}|^2),
\end{align}
we get
\begin{align}
\nonumber
E|x^1(t)|^2\leq& 4E|x^0(t)|^2+4E\mid\int_0^t{f(t,s,x^0(s))}ds\mid^2\\
\nonumber
+&4E\mid\int_0^t{g(t,s,x^0(s))}dW_{s}\mid^2\\
\nonumber
+&4E\mid\int_0^t\int_{R_{0}}{h(t,s,x^0(s),\xi)}\tilde{N}(ds,d\xi)\mid^2.
\end{align}
\\ Employing Cauchy--Schwarz inequality, and It\^{o} isometry, we obtain
\begin{align}
\nonumber
E|x^1(t)|^2\leq& 4E|\varphi(t)|^2+4E\int_0^t{t|f(t,s,x^{0}(s))|^2}ds\\
\nonumber
+&4E\int_0^t{|g(t,s,x^{0}(s))|^2}ds\\
\nonumber
+&4E\int_0^t\int_{R_{0}}{|h(t,s,x^{0}(s),\xi)|^2}\upsilon(d\xi)ds.
\end{align}
In view of Assumption \textcolor[rgb]{1,0,0}{$H_1$}, and since $\varphi(t)$ is monotonically increasing,  one gets
\begin{align}
\nonumber
E|x^1(s)|^2\leq&4E|\varphi(t)|^2+4C\tilde{T}\int_0^t{(1+E|x^0(s)|^2)}ds\\
\nonumber
\leq&4E|\varphi(T)|^2+4C\tilde{T}(1+E|\varphi(T)|^2)t,
\end{align}

This proves (\ref{3}) in case of $k=1$.\\
Now, argue that (\ref{3}) holds for $k$. Then, inductively, we have for $k+1$
 \begin{align}
\nonumber
E|x^{k+1}(t)|^2\leq& 4E|\varphi(t)|^2+4TE\int_0^t{|f(t,s,x^{k}(s))|^2}ds\\
\nonumber
+&4E\int_0^t{|g(t,s,x^{k}(s))|^2}ds+4E\int_0^t\int_{R_{0}}{|h(t,s,x^{k}(s),\xi)|^2}\upsilon(d\xi)ds\\
\nonumber
\leq&4E|\varphi(t)|^2+4C\tilde{T}\int_0^t{(1+E|x^{k}(s)|^2)}ds\\
\nonumber
\leq&4E|\varphi(s)|^2+4C\tilde{T}\int_0^t{(1+4E|\varphi(T)|^2\sum_{\ell=0}^k{\frac{(4C\tilde{T})^{\ell}}{\ell!}{s} ^{\ell}}+\sum_{\ell=1}^k{\frac{(4C\tilde{T})^{\ell}}{\ell!}{s}^{\ell}})}ds\\
\nonumber
=&4E|\varphi(t)|^2+4C\tilde{T}t\\
\nonumber
+&4E|\varphi(T)|^2\sum_{\ell=1}^{k+1}{\frac{(4C\tilde{T})^{\ell}}{\ell!}t^{\ell}}+\sum_{\ell=2}^{k+1}{\frac{(4C\tilde{T})^{\ell}}{\ell!}t^{\ell}}\\
\nonumber
=&4E|\varphi(T)|^2\sum_{\ell=0}^{k+1}{\frac{(4C\tilde{T})^{\ell}}{\ell!}t^{\ell}}+\sum_{\ell=1}^{k+1}{\frac{(4C\tilde{T})^{\ell}}{\ell!}t^{\ell}}\\
\leq& 4E|\varphi(T)|^2\sum_{\ell=0}^{\infty}{\frac{(4C\tilde{T}^2)^{\ell}}{\ell!}}+\sum_{\ell=1}^{\infty}{\frac{(4C\tilde{T}^2)^{\ell}}{\ell!}}.
\end{align}
Hence, (\ref{3}) holds for all arbitrary $k$. And by Eq.(\ref{3}) one may get,
\begin{align}
\nonumber
E|x^{k}(t)|^2\leq4(1+E|\varphi(T)|^2)e^{4C\tilde{T}^2}.
\end{align}
This proves that  the sequence $\{x^k(t), k=1,2,...\}$ has a uniform bound on $[0,T]$. And this completes the proof of Lemma \textcolor[rgb]{1,0,0}{3.2}.
\end{proof}

\begin{lemma} let Assumptions \textcolor[rgb]{1,0,0}{$H_1-H_2$} are fulfilled, there exists a positive constant $C_3$ such that for all $0\leq t \leq T$, $k,m\geq 1$,
\begin{align}
\nonumber
E\left[\sup_{0\leq s \leq t}|x^{k+m}(s)-x^{k}(s)|^2\right]\leq C_3t.
\end{align}
\end{lemma}
\begin{proof} By Eq.(\ref{2})  and the basic inequality (\ref{4}), we have
\begin{align}
\nonumber
E(\sup_{0\leq s \leq t}|x^{k+m}(s)-x^{k}(s)|^2)\leq&3E\left[\sup_{0\leq s \leq t}|\int_0^s{[f(t,u,x^{k+m-1}(u))-f(t,u,x^{k-1}(u))]}du|^2\right]\\
\nonumber
+&3E\left[\sup_{0\leq s \leq t}|\int_0^s{g(t,u,x^{k+m-1}(u))-g(t,u,x^{k-1}(u))]}dW_u|^2\right]\\
\nonumber
+&3E\left[\sup_{0\leq s \leq t}|\int_0^s{h(t,u,x^{k+m-1}(u),\xi)-h(t,u,x^{k-1}(u),\xi)]}\tilde{N}(du,d\xi)|^2\right].
\end{align}
Thanks to Cauchy--Schwarz inequality, Lemma \textcolor[rgb]{1,0,0}{2.2} and  Assumption \textcolor[rgb]{1,0,0}{$H_2$},
\begin{align}
\nonumber
E(\sup_{0\leq s \leq t}|x^{k+m}(s)-x^{k}(s)|^2)\leq&3TE\int_0^t{|f(t,s,x^{k+m-1}(s))-f(t,s,x^{k-1}(s))}|^2 ds\\
\nonumber
+&12E\int_0^t{|g(t,s,x^{k+m-1}(s))-g(t,s,x^{k-1}(s))|^2}ds\\
\nonumber
+&12E\int_0^t{|h(t,s,x^{k+m-1}(s),\xi)-h(t,s,x^{k-1}(s),\xi)|^2}\upsilon(d\xi)ds\\
\nonumber
\leq& 12\tilde{T} \lambda(t)E\int_0^t{\kappa(|x^{k+m-1}(s)-x^{k-1}(s))|^2}ds,
\end{align}
which, with the help of Jensen's inequality  and Lemma \textcolor[rgb]{1,0,0}{3.2}, gives

\begin{align}\label{5}
\nonumber
E(\sup_{0\leq s \leq t}|x^{k+m}(s)-x^{k}(s)|^2)\leq& 12\tilde{T}\lambda(t)\\
\nonumber
\times& \int_0^t{\kappa (E|x^{k+m-1}(s)-x^{k-1}(s)|^2)}ds\\
\leq&C_2\int_0^t{\kappa(4C_1)}ds \leq C_3t,
\end{align}
where $C_2:=12\tilde{T}\lambda(t)$.
\\  Hence Lemma \textcolor[rgb]{1,0,0}{3.3} is obtained.
\end{proof}
We now choose $v\in[0,T]$, $0\leq t \leq v$, such that $\kappa(C_3t) \leq C_3$, where $ \kappa(\eta)= C_2\eta$.
\\ Then, define the following sequences.
\begin{align}\label{6}
\nonumber
&\psi_{1}(t)=C_3t\\
&\psi_{k+1}(t)=\int_0^t{\kappa(\psi_{k}(s))}ds,\;\; k\geq 1\\
\label{7}
&\psi_{k,m}(t)=E(\sup_{0\leq s \leq t}|x^{k+m}(s)-x^{k}(s)|^2), \;\;\;\; k,\;m=1,2,...
\end{align}

\begin{lemma} There exists a positive $t\in[0,v]$ such that, for all $k,m\geq 1$, we have
\begin{align} \label{8}
0 \leq \psi_{k,m}(t)\leq \psi_{k}(t)\leq \psi_{k-1}(t)\leq ...
\leq \psi_{1}(t), \;\;\; \forall t \in[0,v].
\end{align}
\end{lemma}
\begin{proof} By Lemma \textcolor[rgb]{1,0,0}{3.3}, we have
\begin{align}
\nonumber
\psi_{1,m}(t)=E(\sup_{0\leq s \leq t}|x^{1+m}(s)-x^{1}(s)|^2) \leq C_3 t = \psi_{1}(t).
\end{align}
By the definition of $ \kappa$ and equations (\ref{6})-(\ref{7}), we have
\begin{align*}
\nonumber
\psi_{2,m}(t)=&E(\sup_{0\leq s \leq t}|x^{2+m}(s)-x^{2}(s)|^2)\\
\nonumber
\leq&C_2\int_0^t{E(\sup_{0\leq u \leq s}|x^{1+m}(u)-x^{1}(u)|^2)}ds\\
\leq& \int_0^t{\kappa(\psi_{1,m}(s))}ds\\
\leq& \int_0^t{\kappa(\psi_{1}(s))}ds := \psi_{2}(t).
%= \int_0^t{\kappa_1(C_3s)}ds \psi_{2}(t).
\end{align*}
Then, we also have
\begin{align}
\nonumber
\psi_{2}(t)=\int_0^t{\kappa(\psi_{1}(s))}ds \leq \int_0^t{\kappa(C_3s)}ds\leq \int_0^t{C_3}ds=C_3t =\psi_{1}(t).
\end{align}
It has been shown that
\begin{align}
\nonumber
0 \leq \psi_{2,m}(t)\leq \psi_{2}(t)\leq \psi_{1}(t), \;\;\;\; \forall t\in[0,v].
\end{align}
Now, we suppose that (\ref{8}) holds for some $k$. Therefore, using the same inequalities above, yields
\begin{align}
\nonumber
\psi_{k+1,m}(t)\leq& C_2\int_0^t{E\sup_{0\leq u \leq s}|x^{k+m}(u)-x^{k}(u)|^2}ds\\
\nonumber
\leq&\int_0^t{\kappa(\psi_{k,m}(s))}ds\leq \int_0^t{\kappa(\psi_{k}(s))}ds:=\psi_{k+1}(t),\;\; 0\leq t \leq v.
\end{align}
On the other hand, we have
\begin{align}
\nonumber
\psi_{k+1}(t)= \int_0^t{\kappa(\psi_{k}(s))}ds \leq \int_0^t{\kappa(\psi_{k-1}(s))}ds:=\psi_k(t), \;\;\;\; \forall t\in[0,v].
\end{align}
This completes the proof.
\end{proof}

. \\ \textbf{Proof of Theorem \textcolor[rgb]{1,0,0}{3.1}. Existence.} We now argue that
\begin{align}
\nonumber
E(\sup_{0\leq s \leq t}|x^{k+m}(s)-x^{k}(s)|^2)\rightarrow 0, \;\;\; t\in[0,v],
\end{align}
as $k,m \rightarrow \infty$, noticing that $\psi_n$ is continuous on $[0,v]$ and for every $k \geq 1$, $\psi_n(.)$ is decreasing on $[0,v]$.
Furthermore, for each $t$, $\psi_n(t)$ is a decreasing sequence.
Hence, we can define the function $\psi(t)$ by
\begin{align} \label{9}
\psi(t)=\lim_{k\rightarrow\infty} \psi_k(t)=\lim_{k\rightarrow\infty} C_2\int_0^t{\psi_{k-1}(s)}ds= C_2\int_0^t{\psi(s)}ds,
\end{align}
for all $0\leq t \leq v$.
Consequently, since $\psi(t)$ is a continuous function on $[0,v]$, $\psi(0)=0$, then by (\ref{9}) and conditions \textcolor[rgb]{1,0,0}{$H_1-H_2$}, all the conditions of Lemma \textcolor[rgb]{1,0,0}{2.1} are satisfied, hence $\psi(t)=0$ for every $t\in[0,v]$.
Now, from Lemma \textcolor[rgb]{1,0,0}{3.4}, we get
\begin{align}
\nonumber
\psi_{k,m}(t)\leq \sup_{0\leq s \leq v} \psi_k(s)\leq \psi_k(v) \rightarrow 0, \;\;t\in[0,v]
\end{align}
as $ k \rightarrow \infty$, thus, $ \{x^k(t)\}^\infty _{k=1}$ is a Cauchy sequence on $L^2[0,T]$.
Then, by Lemma \textcolor[rgb]{1,0,0}{3.2}, we have
\begin{align}
\nonumber
E|x(t)|^2\leq C_1,
\end{align}
where $C_1$ is a positive constant.

By the above discussion, it is easy to conclude that, for all $t\in [0,v]$,
\begin{align}
\nonumber
&E|\int_0^t{[f(t,s,x^{k}(s))-f(t,s,x(s))]}ds|^2 \rightarrow 0\\
\nonumber
&E|\int_0^t{[g(t,s,x^{k}(s))-g(t,s,x(s))]}dW_s|^2 \rightarrow 0\\
\nonumber
&E|\int_0^t\int_{R_0}{[h(t,s,x^{k}(s),\xi)-h(t,s,x(s),\xi)]}\tilde{N}(ds,d\xi)|^2 \rightarrow 0,
\end{align}
as $k\rightarrow \infty$.
Taking limits on both sides of (\ref{2}), we get
\begin{align}
\nonumber
x(t)=&\varphi(t)+\int_0^t{f(t,s,x(s))}ds+\int_0^t{g(t,s,x(s))}dW_{s}\\
\nonumber
+&\int_0^t\int_{R_{0}}{h(t,s,x(s),\xi)}\tilde{N}(ds,d\xi),
\end{align}
and consequently $x(t)$ is a solution for Eq.~(\ref{1}) on $[0,v]$.
By iteration, the existence of solutions to Eq.~(\ref{1}) can be obtained on $[0,T]$.\\
\\ \textbf{ Uniqueness}: Assume that we have two solutions $ x(t)$ and $y(t)$ to Eq.~(\ref{1}) with $x(0)=y(0)$, and consider the setup from the previous  part, we get
\begin{align}
\nonumber
E|x(t)-y(t)|^2\leq&3E\left[t\int_0^t{|f(t,s,x(s))-f(t,s,y(s))|^2}ds\right]\\
\nonumber
+&3E\left[\int_0^t{|g(t,s,x(s))-g(t,s,y(s))|^2}ds\right]\\
\nonumber
+&3E\left[\int_0^t\int_{R_{0}}{|h(t,s,x(s),\xi)-h(t,s,y(s),\xi)|^2}\upsilon(d\xi)ds\right]\\
\nonumber
\leq&3\tilde{T}E\int_0^t[|f(t,s,x(s))-f(t,s,y(s))|^2+|g(t,s,x(s))-g(t,s,y(s))|^2\\
\nonumber
+&\int_{R_{0}}|h(t,s,x(s),\xi)-h(t,s,y(s),\xi)|^2\upsilon(d\xi)]ds.
\end{align}
Next, applying Jensen's inequality and Assumption \textcolor[rgb]{1,0,0}{$H_2$}, yields
\begin{align}
\nonumber
E|x(t)-y(t)|^2\leq 3\tilde{T}\lambda(t)\int_0^t{\kappa(E|x(s)-y(s))|^2}ds,
\end{align}
since we have $E|x(t)-y(t)|^2=0$ at $t=0$, and $\int_{0^+}^{\infty} \frac{dx}{x}=\infty $. Hence, by Lemma \textcolor[rgb]{1,0,0}{2.1}, we obtain
\begin{align}
\nonumber
E(\sup_{0\leq s \leq t}|x(s)-y(s)|^2)=0, \;\;\;\; \forall \; t\in[0,T].
\end{align}
Therefore, $x(t) = y(t)$, for all $t\in[0,T]$,
which proves uniqueness.\\
Thus, the proof of Theorem \textcolor[rgb]{1,0,0}{3.1} is completed.$\;\;\;\;\;\;\;\;\;\;\;\;\;\;\;\;\;\;\;\;\;\;\;\;\;\;\;\;\;\;\;\;\;\;\;\;\;\;\;\;\;\;\;\;\;\;\;\;\;\;\;\;\;\;\;\;\;\;\;\;\;\Box$

\begin{remark} In Theorem \textcolor[rgb]{1,0,0}{3.1}, under non--Lipschitz condition, we prove that the SVIEs have unique solutions. Moreover, in view of assumption \textcolor[rgb]{1,0,0}{$H_2$}, if we set $\lambda(t)\kappa(|x|)=L.|x|$, for some positive constant L, then in particular, we see that the Lipschitz condition is a special case of our proposed condition, In other words, we have generalised the existence and uniqueness results for SVIEs in []. 
\end{remark}
\begin{Example} \emph{Let us suppose that $W_{t} $ is a scalar Brownian motion, and $\tilde{N}(dt,d\xi)$ is a Poisson random measure with $\sigma$ --finite measure $\upsilon(d\xi)=\lambda f(\xi)d\xi$,  where $\lambda=2$ is the jump rate and $f(\xi) = \frac{1}{\sqrt{2\pi\xi}}\exp({-\frac{(\ln\xi)^{2}}{2}})$, $0< \xi< \infty$.
Note that, $W_{t} $ and $\tilde{N}(dt,d\xi)$ are assumed  to be independent.
\\ Consider the following SIVE 
\begin{align} \label{9}
\nonumber
x(t)=\varphi(t)+&\frac{1}{2}\int_0^t{x(s)}ds+4\int_0^t{cos^2(t-s)x(s)}dW_{s}\\
+&c\int_0^t\int_{R_{0}}{\xi^2 x(s)}\tilde{N}(ds,d\xi),\;\; \varphi(t)=1, \;\;\; t \geq 0.
\end{align}
Here $f(t,s,x) = \frac{1}{2}x$, $g(t,s,x) = 4cos^2(t-s)x$ and $h(t,s,x,\xi)=c\xi^2x$, $c>0$.
\\ Obviously, the conditions of Theorem  \textcolor[rgb]{1,0,0}{3.1} is satisfied. Then, the SVIE (\ref{9}) has a unique solution.}
\end{Example}

 \textbf{Funding}
\\ This study was financed in part by the National Natural Science Foundation of China under Grant No. 11531006.


\begin{thebibliography}{99}
\bibitem[1]{Nacira:2018}
N.Agram, B.{\O}ksendal,  S.Yakhlef,
\newblock New approach to optimal control of stochastic Volterra integral equations,
\newblock { Stochastics}  91 (2019) 1--22. 
\bibitem[2]{Albeverio:2010} S.Albeverio, Z.Brze\'{z}niak, J.L.Wu, Existence of global solutions and invariant measures for stochastic differential equations driven by Poisson type noise with non-Lipschitz coefficients, Journal of Mathematical Analysis and Applications 371 (2010) 309--322.
\bibitem[3]{Applebaum:2009} D.Applebaum, L\'{e}vy Processes and Stochastic Calculus, seconded., Cambridge University Press, Cambridge, UK, 2009.    
\bibitem[4]{Atkinson:1997}
K.E.Atkinson,
\newblock The Numerical Solution of Integral Equations of the Second Kind,
\newblock Cambridge University Press, Cambridge, 1997.
\bibitem[5]{AA}
I.Bihari,
\newblock `A generalization of a lemma of Bellman and its application to uniqueness problem of differential equations'.
\newblock {Acta Mathematica Academiae Scientiarum Hungarica} 7 (1956) 71$-$94.
\bibitem[6]{Bjork:1997}
T.Bj\"{o}rk, Y.Kabanov, W.J.Runggaldier,
\newblock Bond market structure in the presence of marked point processes,
\newblock Mathematical Finance 7 (1997) 211$-$239.
\bibitem[7]{Brunner:1982}
H.Brunner,
\newblock A survey of recent advances in the numerical treatment of Volterra integral and integro-differential equations,
\newblock   Journal of Computational and Applied Mathematics, 8 (3) (1982) 213$-$229.
\bibitem[8]{Brunner:2004}
%\\textit
H.Brunner,
\newblock Collocation Methods for Volterra Integral and Related Functional Differential Equations,
\newblock Cambridge University Press, Cambridge, 2004.

\bibitem[9]{Elliott:1982}
R.J.Elliott,
\newblock Stochastic Calculus and Applications,
\newblock Springer-Verlag, Berlin, 1982.
\bibitem[10]{Geman:2006}
H.Geman, A.Roncoroni,
\newblock Understanding the fine structure of electricity prices,
\newblock  Journal of Business 79 (3) (2006) 1225$-$1261.
\bibitem[11]{Glasserman:1997}
P.Glasserman, S.G.Kou,
\newblock The term structure of simple forward rates with jump risk,
\newblock  Mathematical Finance 13 (2003) 383$-$4410.
\bibitem[12]{Anas} A.D.,Khalaf, Wang, X.J., Impulsive stochastic Volterra integral equations driven by L\'{e}vy noise. Bulletin of the Iranian Mathematical Society, to appear.
\bibitem[13]{Anas2} A.D.,Khalaf, Abouagwa M., Almushaira, M., Wang, X.J.: Stochastic Volterra integral equations with
jumps and the Strong superconvergence of Euler–Maruyama approximation. Journal of Computational and
Applied Mathematics. to appear.
\bibitem[14]{Lamm}
P.K.Lamm,
\newblock `A Survey of Regularization Methods for First-Kind Volterra Equations', in: Surveys on Solution Methods for Inverse Problems.
\newblock Springer-Verlag, Berlin, 2000.
\bibitem[15]{Liang Hui:2017}
H.Liang, Z.W.Yang,  J.F.Gao,
\newblock Strong superconvergence of the Euler Maruyama method for linear stochastic Volterra integral equations,
\newblock  Journal of Computational and Applied Mathematics 317 (2017) 447$-$457.
\bibitem[16]{Mandrekar:2015}  V.Mandrekar, B.R\"{u}diger, Stochastic Integrals with Respect to Compensated Poisson Random Measures. In: Stochastic Integration in Banach Spaces. Probability Theory and Stochastic Modelling, Springer, Cham, 2015.
\bibitem[17]{Maruyama:1955}
G.Maruyama,
\newblock `Continuous Markov processes and stochastic equations'.
\newblock {Rendiconti del Circolo Matematico di Palermo} 4 (1955) 48$-$90.
\bibitem[18]{Oksendal:2007}
B.{\O}ksendal, A.Sulem,
\newblock Applied Stochastic Control of Jump Diffusions. Second ed.
\newblock Springer-Verlag, Berlin, 2007.
\bibitem[19]{Platen:2010}
E.Platen, N.Bruti-Liberati,
\newblock Numerical Solution of Stochastic Differential Equations with Jumps in Finance,
\newblock Springer-Verlag, Berlin, 2010.
\bibitem[20]{Ren:2011} J.Ren, J.Wu, Multi-valued stochastic differential equations driven by Poisson Point process, Progress in Probability 65 (2011) 191--205.
\bibitem[21]{Rudiger:2004} B.R\"{u}diger,  Stochastic integration with respect to compensated Poisson random measures on separable Banach spaces, Stochastics and Stochastic Reports 76 (3) (2004) 213--242.
\bibitem[22]{Schonbucher:2003}
P.J.Sch\"{o}nbucher,
\newblock Credit Derivatives Pricing Models,
\newblock John Wiley and Sons, 2003.
\bibitem[23]{Szynal:1988}
D.Szynal, S.Wedrychowicz,
\newblock On solutions of a stochastic integral equation of the Volterra type with applications for chemotherapy,
\newblock  Journal of Applied Probability 25 (2) (1988) 257$-$267.
\bibitem[24]{Tsokos:1974}
C.P.Tsokos, W.J.Padgett,
\newblock  Random Integral Equations with Applications to Life Sciences and Engineering,
\newblock Academic Press, 1974.
\end{thebibliography}
\end{document}